\documentclass[reqno]{amsart}

\usepackage{enumerate}
\usepackage{amssymb}

\usepackage[colorlinks=true]{hyperref}
\hypersetup{urlcolor=blue, citecolor=cyan}
\hypersetup{citecolor=blue}

\usepackage[latin9]{inputenc}

\numberwithin{equation}{section}

\newtheorem{thm}{Theorem}[section]

\theoremstyle{definition}
\newtheorem{rem}[thm]{Remark}

\newcommand\R{{\mathbb R}}
\newcommand\C{{\mathbb C}}

\newcommand\Tma{T_{\mathrm{max}}}
\newcommand\Comp{{\mathrm{c}}}
\newcommand\Eqdef{\stackrel{\text{\tiny def}}{=}}

\newcommand\Ens{{\mathcal E}}
\newcommand\Ical{{\mathcal I}}
\newcommand\dist{{\mathrm{d}}}

\newcommand\goto{\mathop{\longrightarrow}}

\newcommand\MScN[1]{\href{http://www.ams.org/mathscinet-getitem?mr=#1}{\nolinkurl{(#1)}}}
\newcommand\DOI[1]{\href{http://dx.doi.org/#1}{(doi: \nolinkurl{#1})}}
\newcommand\LINK[1]{\href{#1}{(link: \nolinkurl{#1})}}

\newcommand\DI{u_0 }
\newcommand\DIb{v_0 }

\begin{document}

\title{A Fujita-type blowup result and low energy scattering for a nonlinear Schr\"o\-din\-ger equation}

\def\shorttitle{A Fujita-type blowup result and low energy scattering}

\author[T. Cazenave]{Thierry Cazenave$^1$}
\address{$^1$Universit\'e Pierre et Marie Curie \& CNRS, Laboratoire Jacques-Louis Lions,
B.C. 187, 4 place Jussieu, 75252 Paris Cedex 05, France}
\email{\href{mailto:thierry.cazenave@upmc.fr}{thierry.cazenave@upmc.fr}}

\author[S. Correia]{Sim\~ao Correia$^2$}
\address{$^2$Centro de Matem\'atica e Aplica\c c\~oes Fundamentais, Universidade de Lisboa,
Avenida Prof. Gama Pinto 2, 1649--003 Lisboa, Portugal}
\email{\href{mailto:sfcorreia@fc.ul.pt}{sfcorreia@fc.ul.pt}}

\author[F. Dickstein]{Fl\'avio Dickstein$^3$}
\address{$^3$Instituto de Matem\'atica, Universidade Federal do Rio de Janeiro, Caixa Postal 68530, 21944--970 Rio de Janeiro, R.J., Brazil}
\email{\href{mailto:flavio@labma.ufrj.br}{flavio@labma.ufrj.br}}

\author[F. B.~Weissler]{Fred B.~Weissler$^4$}
\address{$^4$Universit\'e Paris 13,  Sorbonne Paris  Cit\'e, CNRS UMR 7539 LAGA, 99 Avenue J.-B. Cl\'e\-ment, F-93430 Villetaneuse, France}
\email{\href{mailto:weissler@math.univ-paris13.fr}{weissler@math.univ-paris13.fr}}

\thanks{Research supported by the ``Brazilian-French Network in Mathematics"}
\thanks{Fl\'avio Dickstein  was partially supported by CNPq (Brasil), and by the Fondation Sciences Math\'ematiques de Paris.}
\thanks{Sim\~ao Correia was partially supported by FCT (Portugal) through the grant SFRH/BD/96399/2013.}



\begin{abstract}
In this paper we consider the nonlinear Schr\"o\-din\-ger equation $i u_t +\Delta u +\kappa   |u|^\alpha u=0$. We prove that if $\alpha <\frac {2} {N}$ and $\Im \kappa <0$, then every nontrivial $H^1$-solution blows up in finite or infinite time. 
In the case $\alpha >\frac {2} {N}$ and $\kappa \in \C$, we improve the existing low energy scattering results in dimensions $N\ge 7$. More precisely, we prove that if $ \frac {8} {N + \sqrt{ N^2 +16N }} < \alpha \le \frac {4} {N} $, then small data give rise to global, scattering solutions in $H^1$.
\end{abstract}

\maketitle

\section{Introduction}

The main purpose of this article is to prove a Fujita-type blowup result for the nonlinear Schr\"o\-din\-ger equation 
\begin{equation} \label{MNLS} 
i u_t +\Delta u +\kappa   |u|^\alpha u=0.
\end{equation} 
Given an initial value $\DI$, the Cauchy problem for~\eqref{MNLS} has the equivalent  form
\begin{equation} \label{MNLS3} 
u(t)= e^{it \Delta }\DI + i\kappa \int _0^t e^{i(t-s) \Delta }( |u|^\alpha u)(s)\, ds.
\end{equation} 
As is well known, the Cauchy problem~\eqref{MNLS3} is locally well-posed in $H^1 (\R^N ) $ provided $\alpha <\frac {4} {N-2}$. (See~\cite{Kato1}.) More precisely, given $\DI\in H^1 (\R^N ) $, there exist a maximal existence time $\Tma = \Tma (\DI) \in (0, \infty ]$ and a unique solution $u\in C([0, \Tma) , H^1 (\R^N ) )$ of~\eqref{MNLS3}. Moreover, if $\Tma <\infty $, then $u$ blows up at $\Tma $ in the sense that $ \| u(t) \| _{ H^1 }\to \infty $ as $t \uparrow \Tma$. 

Recall that Fujita~\cite{Fujita} proved that if $\alpha < \frac {2} {N}$, then all positive solutions of the nonlinear heat equation
\begin{equation} \label{NLH} 
 u_t = \Delta u +  |u|^\alpha u
\end{equation} 
on $\R^N $ blow up in finite time. In addition, if $\alpha >\frac {2} {N}$, then for initial values sufficiently small in an appropriate sense, the corresponding solution of~\eqref{NLH} is global in time. See~\cite{Fujita}. 
In the intervening years, this classical result has lead to an extensive literature, see the two survey articles~\cite{Levine1, DengL}.
However, the extensions have always been to parabolic equations. 

It turns out that there is a similar blowup dichotomy for the nonlinear Schr\"o\-din\-ger equation~\eqref{MNLS}. The blowup part of this dichotomy concerns the case $\Im \kappa  <0$. Indeed, if $\kappa  \in \R$, in which case~\eqref{MNLS}  becomes the standard nonlinear Schr\"o\-din\-ger equation, well-known energy estimates imply that if $\alpha <\frac {4} {N}$, then all  $H^1$-solutions are global in time and remain bounded. These arguments yield the same result if $\Im \kappa  >0$. On the other hand, we prove that if $\Im \kappa  <0$ and $\alpha <\frac {2} {N}$, there is no global, nontrivial solution of~\eqref{MNLS} that remains bounded in $H^1 (\R^N ) $. 
More precisely, we prove the following result.

\begin{thm} \label{fUZB2} 
Assume $\Im \kappa <0$ and $\alpha <\frac {2} {N}$. It follows that there exists $\delta >0$ such that if $u \not \equiv 0$ is a global $H^1$ solution of~\eqref{MNLS}, then
\begin{equation} \label{fUZB2:1} 
 \sup  _{ 0\le s\le t } \|\nabla u ( s ) \| _{ L^2 } \ge \delta  \| u(0)\| _{ L^2 }^{\frac {N+2} {N}} t ^{\frac {2-N\alpha } {N\alpha }} 
\end{equation} 
for all $t>0$. Moreover,
\begin{equation} \label{fEqUB1:1} 
t^{-\frac {2-N\alpha } {N\alpha }} \sup _{ 0\le s\le t }  \|\nabla u (s) \| _{ L^2 }  \goto _{ t\to \infty  }\infty . 
\end{equation} 
\end{thm} 

In other words, every nontrivial $H^1$-solution blows up in finite or infinite time. 
We expect that blowup in fact occurs in finite time. 

Concerning the global existence part of the dichotomy, it is natural to conjecture that if $\alpha >\frac {2} {N}$, then initial values which are sufficiently small in some norm lead to global solutions which remain bounded and have scattering states in $H^1$. 
This is in fact known in dimension $N=1,2,3$. (See~\cite{CazenaveW1, GinibreOV, NakanishiO}.) In higher dimension $N\ge 4$, 
the best available result seems to be global existence and scattering for small data (i.e. low energy scattering) when 
$N\alpha + 2\alpha  + 2\alpha ^2 >  4$, i.e. $\alpha > \alpha _1$ where 
\begin{equation} \label{fAl1} 
\alpha _1 = \frac {8} {N+2+ \sqrt{ N^2+4 N +36 } } 
\end{equation} 
 (See~\cite{GinibreOV, NakanishiO}.) 

The contribution of this paper to the case $\alpha >\frac {2} {N}$ is that we improve the condition  $\alpha >\alpha _1$ when $N\ge 7$, to $\alpha  > \alpha _2$ with
\begin{equation} \label{fAl2} 
\alpha _2 = \frac {8} {N + \sqrt{ N^2 +16N }} 
\end{equation} 
Our result in this case is the following.

\begin{thm} \label{eGE1} 
Set $X= H^1 (\R^N ) \cap L^2 (\R^N ,  |x|^2 dx)$ equipped with its natural norm.
Let $\kappa \in \C$ and assume
\begin{equation} \label{eGE1:1} 
N\ge 3, \quad \alpha _2 < \alpha  < \frac {4} {N}
\end{equation} 
where $\alpha _2$ is given by~\eqref{fAl2}. 
Let $\DI \in X$ satisfy $\DIb \in H^2 (\R^N ) $ and $ |\cdot | \DIb \in H^1 (\R^N ) $, where $\DIb (x)= e^{i \frac { |x|^2} {4}} \DI (x)$. 
If $  \| \DIb \| _{ H^2 } $ is sufficiently small, then the solution of~\eqref{MNLS3} is global. Moreover, there exists $u^+\in X$ such that $e^{-it \Delta }u(t) \to u^+ $ in $X$ as $t\to \infty $. 
\end{thm} 

 Note  that $\alpha _2< \frac {4} {N}$ and behaves as $\frac {4} {N}$ as $N\to \infty $, not as $\frac {2} {N}$.
Thus there is still a significant gap in high dimensions between the conjecture and the known results.

A fundamental technical tool used in the proofs of the above cited results~\cite{CazenaveW1, GinibreOV, NakanishiO} is the Strichartz inequalities. 
These inequalities involve space-time integrals, where the pair of Lebesgue indices satisfy a certain relationship. 
Usually, the pairs of Lebesgue indices  are  {\it admissible} (see~\cite[p.~808]{CazenaveW2}),
and in particular the low energy scattering results of~\cite{CazenaveW1, GinibreOV, NakanishiO} use admissible pairs. 
Strichartz estimates with non-admissible pairs first appeared in~\cite[Lemma~2.1]{CazenaveW1}, but were not used there for low energy scattering. 
They have subsequently been developed in~\cite{Kato2, Foschi, Vilela}. 
In this paper, we use Strichartz estimates with non-admissible pairs along with the low energy scattering argument of~\cite{CazenaveW1}. This combination enables us to prove low energy scattering for $\alpha >\alpha _2$. 

It is worth noting that the exponent $\alpha =\frac {2} {N}$ is also the critical exponent related to scattering of solutions of~\eqref{MNLS}.  When $\alpha <\frac {2} {N}$ in dimension $N\ge 2$,  it is known that no nonzero solution of~\eqref{MNLS} can be global and scatter in $L^2 (\R^N ) $. 
(See Strauss~\cite{Strauss}, Theorem~3.2 and Example~3.3, p.~68.) 

Theorems~\ref{fUZB2} and~\ref{eGE1} are proved respectively in Sections~\ref{Blowup} and~\ref{SmallData} below. 

\section{Blowup} \label{Blowup} 

The remarkable feature of~\eqref{MNLS} is the identity
\begin{equation} \label{fN1} 
\frac {1} {2} \frac {d} {dt} \int  _{ \R^N  } |u|^2 = -\Im \kappa  \int  _{ \R^N  }  |u|^{\alpha +2}
\end{equation}  
which holds for all $0\le t< \Tma$. (When $\Im \kappa =0$, this is the conservation of charge for the standard NLS.) 
We observe that, were the equation set on a bounded domain with Dirichlet boundary conditions, equation~\eqref{fN1} together with H\"older's inequality would imply that no $H^1$ solution can be global (for all $\alpha >0$), when $\Im \kappa  <0$.

\begin{proof} [Proof of Theorem~$\ref{fUZB2}$]
Let $u(t) \not  \equiv 0$ be a global $H^1$ solution of~\eqref{MNLS}. 
The idea of the proof is to multiply equation~\eqref{MNLS} by a cut-off function, so that the $L^2$ norm can be controlled by the $L^{\alpha +2}$ norm.
We fix the cut-off function $\psi (x)= \nu  \theta ( |x|)$, where
\begin{equation} \label{fEZq12:1}  
\theta (r)= 
\begin{cases} 
1 & 0\le r\le 1 \\
2-r & 1\le r\le 2 \\
0 & r\ge 2
\end{cases} 
\end{equation} 
and $\nu\in \R$ is such   $\|\psi \| _{ L^2 } =1$. Given $\lambda >0$, set
\begin{equation} \label{fEZq12b:1}  
\varphi _\lambda (x)=  \psi  (\lambda x).
\end{equation} 
It follows in particular that $\varphi _\lambda  \in C_\Comp (\R^N ) \cap W^{1, \infty }( \R^N )$, $\varphi  _\lambda \ge 0$, 
\begin{equation} \label{fEZq12b2:1}  
 \|\varphi _\lambda  \| _{ L^2 } =  \lambda ^{-\frac {N} {2}} \quad  \text{and}\quad   \| \nabla \varphi _\lambda \| _{ L^\infty  } = \nu\lambda  .
\end{equation}  
Multiplying equation~\eqref{MNLS} by $\varphi _\lambda ^2 \overline{u} $ and taking the imaginary part, we obtain
\begin{equation} \label{fEZq4:1} 
\frac {1} {2} \frac {d} {dt} \int  _{ \R^N  }   |u|^2 \varphi _\lambda  ^2 =2 \Im \int  _{ \R^N  }  \overline{u} \varphi _\lambda \nabla u \cdot \nabla \varphi _\lambda -\Im \kappa   \int  _{ \R^N  } |u|^{\alpha +2} \varphi _\lambda ^2.
\end{equation} 
(To be precise, the equation makes sense in $H^{-1}$, so we take the $H^{-1}-H^1$ duality bracket of the equation with $\varphi _\lambda ^2 \overline{u} \in H^1$.)
Set 
\begin{equation} \label{fEZq4b:1} 
f_\lambda (t)=   \| u \varphi _\lambda \| _{ L^2 }
\end{equation} 
and
\begin{equation}  \label{fEZq6:1} 
K_t =   \| \nabla u\| _{ L^\infty ((0,t), L^2) } .
\end{equation} 
It follows from H\"older's inequality and~\eqref{fEZq12b2:1}  that
\begin{equation*} 
f_\lambda (t)^{\alpha +2} \le  \| \varphi _\lambda \| _{ L^2 }^\alpha    \int  _{ \R^N  } |u|^{\alpha +2} \varphi _\lambda ^2  =  \lambda ^{-\frac {N\alpha } {2}} \int  _{ \R^N  } |u|^{\alpha +2} \varphi _\lambda ^2  ,
\end{equation*} 
so that
\begin{equation}  \label{fEZq6:2} 
 \int  _{ \R^N  } |u|^{\alpha +2} \varphi _\lambda ^2 \ge  \lambda ^{ \frac {N\alpha } {2}}  f_\lambda (t)^{\alpha +2}  .
\end{equation} 
Moreover, we deduce from H\"older's inequality, \eqref{fEZq12b2:1}, \eqref{fEZq4b:1} and~\eqref{fEZq6:1}  that
\begin{equation}   \label{fEZq6:2b} 
 \Bigl| \Im \int  _{ \R^N  }  \overline{u} \varphi _\lambda \nabla u \cdot \nabla \varphi _\lambda  \Bigr| \le \nu\lambda K_t  f_\lambda (t).
\end{equation} 
Consequently, \eqref{fEZq4:1},  \eqref{fEZq6:2} and~\eqref{fEZq6:2b} yield
\begin{equation} \label{fEZq7:1} 
f_\lambda ' \ge - 2 \nu\lambda K_t -\Im \kappa  \lambda ^{\frac {N\alpha } {2}} f_\lambda ^{\alpha +1} \ge  - 2\nu \lambda K_T -\Im \kappa  \lambda ^{\frac {N\alpha } {2}} f_\lambda ^{\alpha +1},
\end{equation} 
for all $0<t\le T<\infty $, where we used the property that $K_t$ is nondecreasing in $t$ in the last inequality. 
Therefore, if 
\begin{equation}  \label{fEZq8:1} 
 f_\lambda (0) ^{\alpha +1} \ge  4 (-\Im \kappa )^{-1} \lambda ^{\frac {2-N\alpha } {2}}  \nu K_T,
\end{equation} 
it follows that $f_\lambda $ is increasing on $(0,T)$, and
\begin{equation} \label{fEZq10:1} 
f_\lambda ' \ge \frac {-\Im \kappa } {2} \lambda ^{\frac {N\alpha } {2}} f_\lambda ^{\alpha +1} 
\end{equation} 
on $(0,T)$. 
Equation~\eqref{fEZq10:1} implies that $f_\lambda $ must blow up before the finite time $\frac {2} {-\Im \kappa \alpha  }  \lambda ^{ - \frac {N\alpha } {2}}  f_\lambda (0) ^{-\alpha }$.
Therefore,
\begin{equation} \label{fEZq9:1} 
T \le \frac {2} {-\Im \kappa \alpha  }  \lambda ^{ - \frac {N\alpha } {2}}   f_\lambda (0) ^{-\alpha } .
\end{equation} 
Note that $f_ \lambda (0)$ is a nonincreasing function of $\lambda >0$ and
\begin{equation} \label{fEZq16:1} 
f_\lambda (0)=  \Bigl( \int  _{ \R^N  }  |u(0, x )|^2 \psi ^2(\lambda x)\, dx \Bigr)^{\frac {1} {2}} \goto
\begin{cases} 
0 &  \text{as }\lambda \uparrow \infty \\ 
 \| u(0, \cdot ) \| _{ L^2 } & \text{as } \lambda \downarrow 0 .
\end{cases} 
\end{equation} 
We first show that
\begin{equation} \label{fEZq17:1} 
K_T \goto _{ T\to \infty  }\infty .
\end{equation} 
Indeed, suppose by contradiction that $K_T$ is bounded as $T\to \infty $. It follows from~\eqref{fEZq16:1} that we can choose $ \widetilde{\lambda} >0$ sufficiently small so that~\eqref{fEZq8:1} with $\lambda = \widetilde{\lambda}$ holds for all $T>0$. We deduce from~\eqref{fEZq9:1}  that $T \le \frac {2} {-\Im \kappa \alpha  }  \widetilde{\lambda} ^{ - \frac {N\alpha } {2}}   f_{\widetilde{\lambda}} (0) ^{-\alpha }$ for all $T>0$. This is absurd, proving~\eqref{fEZq17:1}. 

We next prove~\eqref{fUZB2:1}. Fix $\lambda _0>0$ such that
\begin{equation}  \label{fEZq17b:1} 
f _{ \lambda _0 } (0) = \frac {1} {2}  \| u(0, \cdot ) \| _{ L^2 }.
\end{equation} 
Note that this is possible by~\eqref{fEZq16:1}. 
It follows from~\eqref{fEZq17:1} that if $T>0$ is sufficiently large, $\lambda =\lambda (T)$ defined by 
\begin{equation}  \label{fEZq18:1} 
 f _{ \lambda _0 } (0) ^{\alpha +1} = 4 (-\Im \kappa  )^{-1} \lambda (T) ^{\frac {2-N\alpha } {2}}  K_T,
\end{equation} 
satisfies
\begin{equation}  \label{fEZq19:1} 
\lambda (T)\le \lambda _0.
\end{equation} 
Since $f_\lambda (0)$ is a nonincreasing function of $\lambda $, we deduce from~\eqref{fEZq18:1}-\eqref{fEZq19:1} that~\eqref{fEZq8:1} holds with $\lambda =\lambda (T)$. Therefore, it follows from~\eqref{fEZq9:1} that
\begin{equation} \label{fEZq20:1} 
T \le \frac {2} {-\Im \kappa \alpha  }  \lambda (T) ^{ - \frac {N\alpha } {2}}   f_{\lambda (T)} (0) ^{-\alpha } .
\end{equation} 
Using again~\eqref{fEZq19:1}, we deduce from~\eqref{fEZq20:1} that
\begin{equation} \label{fEZq21:1} 
T \le \frac {2} {-\Im \kappa \alpha  }  \lambda (T) ^{ - \frac {N\alpha } {2}} f_ { \lambda _0 } (0) ^{-\alpha } .
\end{equation} 
Since~\eqref{fEZq18:1} implies
\begin{equation} \label{fEZq21:2} 
 \lambda (T) ^{- \frac {2-  N\alpha } {2}}   =  \frac {4} {-\Im \kappa } K_T f _{ \lambda _0 } (0) ^{- (\alpha +1)} 
\end{equation} 
formulas~\eqref{fEZq21:1} and~\eqref{fEZq21:2}  yield
\begin{equation} \label{fEZq22:1} 
\begin{split} 
T ^{\frac {2-N\alpha } {N\alpha }} & \le  \frac {4} {-\Im \kappa   } \Bigl(  \frac {2} {-\Im \kappa \alpha  } f _{ \lambda _0 } (0) ^{-\alpha } \Bigr)^{\frac {2-N\alpha } {N\alpha }}  f _{ \lambda _0 } (0) ^{- (\alpha +1 ) }    K_T \\ & = 2^{ \frac {2+N\alpha } {N\alpha }} (-\Im \kappa  )^{- \frac {2} {N\alpha }} \alpha ^{- \frac {2-N\alpha } {N\alpha }}    f _{ \lambda _0 } (0) ^{- \frac {N+2 } {N }}        K_T.
\end{split} 
\end{equation} 
Inequality~\eqref{fUZB2:1} now follows from~\eqref{fEZq22:1} and~\eqref{fEZq17b:1}.

We finally prove~\eqref{fEqUB1:1}. 
Given $T>0$, it follows from \eqref{fEZq16:1} that there exists a unique $\mu (T)>0 $ such that  
\begin{equation}  \label{fEqUB2:1} 
 f _{ \mu (T) } (0) ^{\alpha +1} =  4 (-\Im \kappa  )^{-1} \mu (T) ^{\frac {2-N\alpha } {2}}  K_T.
\end{equation} 
Since $ K_T$ is a  nondecreasing function of $T$, it follows from \eqref{fEqUB2:1} that the map $T\mapsto f _{ \mu (T) } (0) ^{\alpha +1} \mu (T) ^{-{\frac {2-N\alpha } {2}}} $ is also nondecreasing. 
On the other hand, $f_\mu (0)$ is a nonincreasing function of $\mu $, so that the map $\mu \mapsto f _{ \mu } (0) ^{\alpha +1} \mu ^{-{\frac {2-N\alpha } {2}}}$ is decreasing, so we conclude that the map $T\mapsto \mu(T)$ is nonincreasing.
Since $f_\mu (0) \le  \| u(0)\| _{ L^2 }$ for all $\mu >0$ by~\eqref{fEZq16:1} 
and $K_T \ge  \delta  \| u(0)\| _{ L^2 }^{\frac {N+2} {N}} T ^{\frac {2-N\alpha } {N\alpha }}$ by~\eqref{fUZB2:1} we deduce from~\eqref{fEqUB2:1} that
\begin{equation} \label{fEZ1} 
\mu (T) \le (4 (-\Im \kappa  )^{-1} \delta  )^{-\frac {2} {2-N\alpha }} \| u(0) \| _{ L^2 }^{-\frac {2} {N}} T^{-\frac {2} {N\alpha }} \goto _{ T\to \infty  }0.
\end{equation} 
We deduce in particular from~\eqref{fEZq16:1} and~\eqref{fEZ1} that
 $f _{ \mu (T) } (0) \to  \| u(0) \| _{ L^2 }$ as $T\to \infty $ so that by~\eqref{fEqUB2:1} 
\begin{equation} \label{fEqUB3:1} 
 \mu (T) ^{\frac {2-N\alpha } {2}}  K_T \goto _{ T\to \infty  }  \frac { \| u(0) \| _{ L^2 }^{\alpha +1} (-\Im \kappa ) } {4}.
\end{equation} 
Moreover, it follows from~\eqref{fEqUB2:1} that~\eqref{fEZq8:1} is satisfied with $\lambda =\mu (T)$, so that~\eqref{fEZq10:1} holds, i.e.
\begin{equation} \label{fEqUB4:1} 
f_{\mu (T)}' \ge \frac {-\Im \kappa  } {2} \mu (T) ^{\frac {N\alpha } {2}} f_{\mu (T)}^{\alpha +1} 
\end{equation} 
for all $0<t<T$.
Integrating the above differential inequality on $(T/2 ,T)$ and using \eqref{fEqUB3:1}, we obtain
\begin{equation}  \label{fEqUB5:1} 
f _{ \mu (T) }(T/2 )^{-\alpha }-f _{ \mu (T) }(T)^{-\alpha }\ge \frac {-\Im \kappa \alpha  } {4}   \mu (T) ^{\frac {N\alpha } {2}} T \ge \eta K_T^{-\frac {N\alpha } {2-N\alpha }} T
\end{equation} 
for $T\ge 2$, with $\eta >0$. 
Next, recall that $\mu (t)$ is a nonincreasing function of $t$,  and that the map $\lambda \mapsto f_\lambda (t) $ is a nonincreasing function of $\lambda $, so that 
\begin{equation} \label{fZZ1} 
f _{ \mu (s) } (\tau ) \le   f_{ \mu (t) } (\tau ) 
\end{equation} 
for all $\tau >0$ and $0<s<t$.
 Therefore, letting $s=\tau =T/2$ and $t=T$, we see that
  $  f _{ \mu (T/2) } (T/2)^{-\alpha } \ge   f_{ \mu (T) } (T/2) ^{-\alpha }$ and it follows from~\eqref{fEqUB5:1}  that
  \begin{equation}  \label{fEqUB5:1b1} 
f _{ \mu (T/2) }(T/2 )^{-\alpha }-f _{ \mu (T) }(T)^{-\alpha } \ge \eta K_T^{-\frac {N\alpha } {2-N\alpha }} T
\end{equation} 
 for $T\ge 2$. 
Next, we deduce from~\eqref{fZZ1} and the fact that the map $\tau \mapsto f _{ \mu (t) } (\tau )$ is increasing on $(0,t)$, that 
\begin{equation}  \label{fEqUB5:52} 
f _{ \mu (s) } (s ) \le   f_{ \mu (t) } (s )  \le   f_{ \mu (t) } (t ) 
\end{equation} 
for all $0<s<t$. Thus we see that the map $t\mapsto f_{ \mu (t) } (t ) $ is nondecreasing; and so the map $t\mapsto f_{ \mu (t) } (t ) ^{-\alpha }$ is nonincreasing, so it has a limit as $t\to \infty $. Letting $T\to \infty $ in~\eqref{fEqUB5:1b1}, we deduce that $K_T^{-\frac {N\alpha } {2-N\alpha }} T \to 0$ as $T\to \infty $, which is the desired result.
\end{proof} 

\begin{rem} 
Under the assumption
\begin{equation} \label{fEWG1:3} 
-\Im \kappa   \ge \frac {\alpha } {2\sqrt{\alpha +1}}  |\Re \kappa  |
\end{equation} 
we may replace $ \sup  _{ 0\le s\le t } \|\nabla u ( s ) \| _{ L^2 }$ by $ \| \nabla u (t)\| _{ L^2 }$ in estimates~\eqref{fUZB2:1} and~\eqref{fEqUB1:1} of Theorem~\ref{fUZB2}.
Indeed, it follows from~\eqref{fEWG1:3} that $  \| \nabla u(t) \| _{ L^2 } $ is a nondecreasing function of $t$. To see this, note that for a solution of~\eqref{MNLS} we have
\begin{equation} \label{fEWG1:1} 
\frac {1} {2} \frac {d} {dt}  \| \nabla u(t)\| _{ L^2 }^2 = \Re  \Bigl( i\kappa  \int  _{ \R^N  } \nabla ( |u|^\alpha u ) \cdot \nabla  \overline{u}  \Bigr) \Eqdef A. 
\end{equation} 
Since 
\begin{equation}  \label{fEq2b2} 
 \nabla ( |u|^{\alpha }u) =  \frac {\alpha +2} {2}   |u|^\alpha \nabla u + \frac {\alpha } {2}  |u|^{\alpha -2} u^2 \nabla  \overline{u}, 
\end{equation} 
we see that
\begin{equation}  \label{fEq2b3} 
 \nabla ( |u|^{\alpha }u)\cdot \nabla  \overline{u}  = \frac {\alpha +2} {2}   |u|^\alpha  |\nabla u|^2 + \frac {\alpha } {2}  |u|^{\alpha -2} u^2  (\nabla  \overline{u})^2 .
\end{equation} 
It follows that 
\begin{equation}  \label{fEq2b6} 
\begin{split} 
\Re [ i\kappa   \nabla ( |u|^{\alpha }u)\cdot \nabla  \overline{u} ] & =  -\Im \kappa    \frac {\alpha +2} {2}   |u|^\alpha  |\nabla u|^2+ \frac {\alpha } {2}  \Re [ i\kappa  |u|^{\alpha -2} u^2  (\nabla  \overline{u})^2 ] 
\\ & \ge  \Bigl(- \frac {\alpha +2} {2} \Im \kappa   - \frac {\alpha } {2} |\kappa | \Bigr)  |u|^\alpha  |\nabla u|^2.
\end{split} 
\end{equation} 
This shows that $A\ge 0$ provided~\eqref{fEWG1:3} holds. 
The above calculations are justified if $u$ is an $H^2$ solution. The result follows by approximation, regularity, and continuous dependence. (All these properties are established in~\cite{Kato1}.)
Note that~\eqref{fEWG1:3}  is identical to condition~(2.2) in~\cite{OkazawaY}.
\end{rem} 

\begin{rem} \label{fUB4} 
Under the assumptions of Theorem~\ref{fUZB2}, we do not know whether or not there exists a global $H^1$ solution of~\eqref{MNLS}. In fact, if  such a solution does exist, it would necessarily have a stronger dispersion than the solutions of the linear Schr\"o\-din\-ger equation. Indeed, suppose  
$u\not \equiv  0$ is a global $H^1$ solution of~\eqref{MNLS}
 and let $R(t)$ satisfy
\begin{equation*} 
\int  _{  \{|x|\le R(t) \} }  |u(t,x)|^2 dx = \frac {1} {2}\int  _{ \R^N  }  |u(t,x)|^2 dx.
\end{equation*} 
We claim that
\begin{equation} \label{fUB4:2} 
\limsup  _{ t\to \infty  } t^{-\frac {2} {N\alpha }} R(t) = \infty .
\end{equation} 
To see this, observe that by H\"older's inequality and the definition of $R(t)$,
\begin{equation} \label{fUB4:3} 
\int  _{ \R^N  }  |u|^{\alpha +2} \ge  2^{-\frac{\alpha +2}{2}} \omega _N^{- \frac {\alpha } {2}} R^{ - \frac {N\alpha } {2} } f(t)^{\frac {\alpha +2} {2}}
\end{equation} 
where $f(t)= \int  _{ \R^N  } |u|^2$ and $\omega _N$ is the measure of the unit ball of $\R^N $. It follows from~\eqref{MNLS}  and~\eqref{fUB4:3} that 
\begin{equation*} 
f' \ge - \Im \kappa  2^{-\frac{\alpha}{2}} \omega _N^{- \frac {\alpha } {2}} R^{ - \frac {N\alpha } {2} } f^{\frac {\alpha +2} {2}}.
\end{equation*} 
Therefore, 
$\int _0^\infty R(t)^{-\frac {N\alpha } {2}}  dt < \infty $, 
which yields~\eqref{fUB4:2}.  
On the other hand, let $ \widetilde{u} (t)= e^{it\Delta } \DI$ where $\DI \in H^1 (\R^N ) $, $\DI \not = 0$.  Multiplying the equation $i \widetilde{u} _t +\Delta  \widetilde{u} =0$ by $\psi _M  \overline{u} $, where $\psi _{M} (x)=  \min \{ \frac {x} {M}, 1 \}$, we obtain
\begin{equation} \label{fUB6:5} 
\int  _{ \R^N  } \psi _M  | \widetilde{u} |^2 \le \int  _{ \R^N  } \psi _M  | \DI |^2 + \frac {2t} {M} \| \DI\| _{ L^2 } \|\nabla \DI\| _{ L^2 }.
\end{equation} 
(Cf.~\cite[Lemma~5.4]{GinibreV}).)
Given $t>0$, we substitute $M= at$ in~\eqref{fUB6:5} with $a= 16\frac {  \| \nabla \DI \| _{ L^2 }} { \| \DI \| _{ L^2 }}$.
Since $\psi _M \ge 1 _{ \{  |x|>M \} }$, this yields
\begin{equation*} 
\int  _{ \{  |x|>at \}  }  | \widetilde{u} |^2 \le \int  _{ \R^N  } \psi _{at}  | \DI |^2 + \frac {1} {8} \| \DI\| _{ L^2 }^2.
\end{equation*} 
Furthermore, $\int  _{ \R^N  } \psi _{at}  | \DI |^2 \to 0$ as $t\to \infty $ by dominated convergence, and so
\begin{equation*} 
\int  _{ \{  |x|>at \}  }  | \widetilde{u} |^2 \le  \frac {1} {4} \| \DI\| _{ L^2 }^2
\end{equation*} 
for $t$ large. Therefore,
\begin{equation} \label{fUB6:7} 
\int  _{ \{  | \widetilde{x} |<at \}  }  |u|^2 \ge  \frac {3} {4}  \| \DI\| _{ L^2 }^2
\end{equation} 
for $t$ large. In particular, if $ \widetilde{R} (t)$ satisfies 
\begin{equation*} 
\int  _{  \{|x|\le  \widetilde{R} (t) \} }  | \widetilde{u} (t,x)|^2 dx = \frac {1} {2}\int  _{ \R^N  }  | \widetilde{u} (t,x)|^2 dx = \frac {1} {2}  \| \DI \| _{ L^2 }^2,
\end{equation*} 
then $ \widetilde{R} (t) \le at$ for $t$ large. Comparing with~\eqref{fUB4:2}, we see that $u$ has a stronger dispersion than $ \widetilde{u} $ as $t\to \infty $.
\end{rem} 

\begin{rem} 
If we look for solutions of~\eqref{MNLS}  of the form 
\begin{equation*} 
u(t,x)= \rho (t) e^{i \frac { |x|^2} {4(t+ t_0)}},
\end{equation*} 
where $t_0>0$ is given, then $\rho $ must satisfy
\begin{equation*} 
\rho '= -\frac {N} {2(t+t _0)} \rho -i\kappa   |\rho| ^{\alpha } \rho .
\end{equation*} 
Setting $z= (t+t_0)^{\frac {N} {2} } \rho $, we get to the equation
\begin{equation*} 
z'=-i\kappa  (t+t_0)^{- \frac {N\alpha } {2}}  |z|^{\alpha }z.
\end{equation*} 
Multiplying the equation by $ \overline{z} $ and taking the real part, one easily gets to
\begin{equation} \label{feqe1}  
\frac {1} {\alpha  |z(t)|^\alpha }= \frac {1} {\alpha  |z(0)|^\alpha } +\Im \kappa  \int _0^t \frac {ds} {(s+t_0)^{\frac {N\alpha } {2}}}.
\end{equation} 
If $\alpha >\frac {2} {N}$, then the integral on the right-hand side of~\eqref{feqe1}  is convergent, and we see that if $ |z(0)|$ is sufficiently small so that 
\begin{equation*} 
 \frac {1} {\alpha  |z(0)|^\alpha } \ge -\Im \kappa   \int _0^\infty  \frac {ds} {(s+t_0)^{\frac {N\alpha } {2}}},
\end{equation*} 
then the solution is global; and if $ |z(0)|$ is larger, then the solution blows up in finite time. On the other hand, if $\alpha \le \frac {2} {N}$, then the integral on the right-hand side of~\eqref{feqe1}  is divergent. Therefore, for every $z(0)$, the solution blows up at the finite time $T$ given by
\begin{equation*} 
 \frac {1} {\alpha  |z(0)|^\alpha } = -\Im \kappa  \int _0^T  \frac {ds} {(s+t_0)^{\frac {N\alpha } {2}}}.
\end{equation*} 
Thus we see that the exponent $\alpha =\frac {2} {N}$ is critical.
\end{rem} 

\section{Low energy scattering} \label{SmallData} 

To prove Theorem~\ref{eGE1}, we first prove a local existence result for small  data for the following equation\begin{equation} \label{MNLS2} 
v(t) = e^{it\Delta }\DIb +  \int _0^t h(s) e^{i(t-s) \Delta }  ( |v|^\alpha v)(s)\, ds
\end{equation} 
where
\begin{equation} \label{fLB3} 
h(t)=  i\kappa  (1-t)^{- \frac {4 - N\alpha } {2}}  
\end{equation} 
for $0\le t<1$.
As we will see, equation~\eqref{MNLS2} is equivalent to equation~\eqref{MNLS} via the pseudo-conformal transformation.
Before stating the result, we introduce some notation. We assume~\eqref{eGE1:1} and we  set 
\begin{equation} \label{fLB8} 
\rho = \frac {N(\alpha +2)} {N+\alpha },\quad \gamma = \frac {4(\alpha +2)} {\alpha (N-2)}
=  \frac {4\rho } {N(\rho -2)} .
\end{equation} 
It is not difficult to show that
\begin{equation}  \label{fLB8b3}
2< \rho < \frac {2N} {N-2}, \qquad N> \rho 
\end{equation} 
and that $(\gamma ,\rho )$ is an admissible pair, i.e. $\frac {2} {\gamma }= N(\frac {1} {2}- \frac {1} {\rho })$ (see~\cite[Proposition~1.5~(ii)]{CazenaveW2}).
Note also that $\alpha >\alpha _2$, which implies $N\alpha ^2 +N\alpha > 4$, so that
\begin{equation} \label{fLB16b2}
\frac  {4- (N-4) \alpha } {2(\alpha +2)} > \frac  {4-N\alpha } {2}
\end{equation} 

\begin{thm} \label{eGE3} 
Suppose $N\ge 3$,  $\kappa \in \C$ and $\alpha _2 < \alpha  < \frac {4} {N}$ where $\alpha _2$ is given by~\eqref{fAl2}. Fix 
\begin{equation}  \label{fDfa2} 
a \ge \gamma 
\end{equation} 
 sufficiently large so that 
\begin{equation} \label{fDfa} 
\frac {4- (N-4) \alpha } {2(\alpha +2)}  - \frac {\alpha } {a} > \frac {4-N\alpha } {2} .
\end{equation} 
(The existence of $a$ is guaranteed by~\eqref{fLB16b2}.) 
There exists $\delta >0$ such that if 
\begin{gather} 
\DIb \in X \label{fFP14} \\
(1+  |\cdot |) e^{it \Delta } \DIb \in  L^a ((0,1), L^\rho ( \R^N ))  \label{fFP15} \\
e^{it \Delta }\nabla  \DIb \in  L^a ((0,1), L^\rho ( \R^N )) \label{fFP16} \\
\| \nabla  e^{it \Delta } \DIb \| _{ L^a ((0,1), L^\rho ) } \le \delta \label{fFP17}
\end{gather} 
 then there exists a solution $v\in C([0,1], X)$ of~\eqref{MNLS2}.
\end{thm} 

\begin{proof}  [Proof of Theorem~$\ref{eGE3}$]
We define $ \widetilde{a} $ by
\begin{equation} \label{fDfat} 
\frac {1} {a}+ \frac {1} { \widetilde{a} }=\frac {2} {\gamma }
\end{equation} 
and we recall the following Strichartz-type estimate for non-admissible pairs 
\begin{equation} \label{fLB5} 
 \Bigl\| \int _0^\cdot e^{i(\cdot -s)\Delta } f(s)\, ds \Bigr\| _{ L^a((0,1), L^\rho ) } \le C  \| f \| _{ L^{ \widetilde{a}' }((0,1), L^{\rho '}) } .
\end{equation}  
(See~\cite[Lemma~2.1]{CazenaveW1}.)
If $\mu $ is defined by
\begin{equation} \label{fDfmu} 
\frac {1} {\mu }= \frac {4- (N-4) \alpha } {2(\alpha +2)}  - \frac {\alpha } {a} < \frac {4- (N-4) \alpha } {2(\alpha +2)}<1
\end{equation} 
then it follows from~\eqref{fDfa} and~\eqref{fLB3} that
\begin{equation} \label{fLB7} 
h\in  L^\mu   (0,1 ) .
\end{equation} 
Next, we deduce from Sobolev's inequality that
\begin{equation} \label{fLB9} 
 \| v \| _{ L^{\frac {\alpha \rho } {\rho -2}} } =  \| v \| _{ L^{\frac {N \rho } {N- \rho }} } \le C  \| \nabla v \| _{ L^\rho  }
\end{equation} 
and so by H\"older's inequality,
\begin{equation} \label{fLB10} 
 \| \,  |v|^\alpha w \| _{ L^{\rho '} }\le  \|v\| ^\alpha _{ L^{\frac {\alpha \rho } {\rho -2}} }  \| w \| _{ L^\rho  }\le C  \|\nabla v\| _{ L^\rho  }^\alpha  \|w\| _{ L^\rho  }. 
\end{equation} 
It easily follows from~\eqref{fLB10} that
\begin{gather}  
 \| \,  |v|^\alpha v\| _{ L^{ \rho '} } \le C \| \nabla v\| _{ L^\rho }^\alpha   \|v\| _{ L^\rho }   \label{fLB11b1} \\
  \| \, |\cdot |\,   |v|^\alpha v\| _{ L^{ \rho '} } \le C \| \nabla v\| _{ L^\rho }^\alpha   \| \,  |\cdot |v\| _{ L^\rho }   \label{fLB11b2} \\
  \| \nabla  ( |v|^\alpha v ) \| _{ L^ {\rho '} } \le C \| \nabla v\| _{ L^\rho }^{\alpha +1}   \label{fLB11} \\
  \| \,  |u|^\alpha u -  |v|^\alpha v\| _{ L^{ \rho '} }\le  C ( \| \nabla u\| _{ L^\rho  }^\alpha +  \| \nabla v \| _{ L^\rho  }^\alpha)  \| u-v \| _{ L^\rho }. 
  \label{fLB12}
\end{gather} 
We observe, by~\eqref{fDfmu}  and~\eqref{fDfat}, that
\begin{equation} \label{fLB11b3} 
\frac {1} { \widetilde{a} '}= \frac {1} {\mu } + \frac {\alpha +1} {a}.
\end{equation} 
Again using H\"older's inequality, but with the time integrals, we deduce from~\eqref{fLB11b3},  
along with respectively~\eqref{fLB11b1},  \eqref{fLB11b2}, \eqref{fLB11} and~\eqref{fLB12},  that
\begin{gather} 
 \| h   |v|^\alpha v \|  _{ L^{ \widetilde{a}' }( (0,1) , L^{ \rho '}) } \le   C  \|h  \| _{ L^\mu  (0,1)}   \| \nabla  v \| _{ L^a( (0,1) , L^\rho  )}^{\alpha }  \|   v \| _{ L^a( (0,1) , L^\rho )}   \label{fLB13b1} \\
  \| h  |\cdot |\,   |v|^\alpha v \|  _{ L^{ \widetilde{a}' }( (0,1) , L^{ \rho '}) } \le   C  \|h  \| _{ L^\mu (0,1)}   \| \nabla  v \| _{ L^a( (0,1) , L^\rho  )}^{\alpha }  \| \, |\cdot |\,   v \| _{ L^a( (0,1) , L^\rho )}   \label{fLB13b2} \\
  \| \nabla (h   |v|^\alpha v ) \|  _{ L^{ \widetilde{a}' }( (0,1) , L^{ \rho '}) } \le   C  \|h  \| _{ L^\mu (0,1)}   \| \nabla  v \| _{ L^a( (0,1) , L^\rho )}^{\alpha +1}    \label{fLB13}
\end{gather} 
and
\begin{multline}  \label{fLB15}
 \| h  ( |u|^\alpha u -  |v|^\alpha v) \|  _{ L^{ \widetilde{a}' }((0,1) , L^{\rho '}) } \le C 
  \|h  \| _{ L^\mu (0,1)} \\ \times  (   \| \nabla u \| _{ L^a( (0,1) , L^\rho )}^{\alpha }+   \| \nabla v \| _{ L^a( (0,1) , L^\rho )}^{\alpha })
    \|u -v \| _{ L^a( (0,1) , L^{\rho } )} .
\end{multline} 
We construct the solution $v$ of~\eqref{MNLS3} by a contraction mapping argument in the set
$\Ens _{ \delta ,M }$  defined for  $\delta ,M>0$ by
\begin{equation} \label{fFP1} 
\begin{split} 
\Ens &_{ \delta ,M } = \{ v\in L^a((0,1), W^{1, \rho } (\R^N )) ;  \,  |\cdot |v\in L^a((0,1), L^{\rho } (\R^N )) ,\\ &
  \|  v \| _{ L^a ((0,1), L^\rho ) }\le M,  \|  |\cdot | v \| _{ L^a ((0,1), L^\rho ) }\le M \text{ and }  \| \nabla v \| _{ L^a ((0,1), L^\rho ) }\le \delta   \}.
\end{split} 
\end{equation} 
We set $\dist (v,w) =  \|v-w\| _{ L^a (0,1), L^\rho ) }$ so that $ ( \Ens _{ \delta ,M }, d)$ is a complete metric space. 
Fix $\DIb\in X $ and, given $v\in \Ens _{ \delta ,M }$, let  $\Ical (v)$ and $\Phi (v)$ be defined by
\begin{gather}
\Ical (v)(t)=  \int _0^t h(s) e^{i(t-s) \Delta }  ( |v|^\alpha v)(s)\,ds \label{fFP2} \\
\Phi ( v )(t)= e^{it\Delta } \DIb + \Ical ( v ) (t) . \label{fFP3} 
\end{gather} 
It follows from~\eqref{fLB13b1}, \eqref{fLB13}, \eqref{fLB15} and the estimate~\eqref{fLB5} that, for some constant $C$ independent of $\delta $, $M$ and $v,w\in \Ens _{ \delta ,M }$,
\begin{gather} 
 \| \Ical (v) \| _{ L^a((0,1), L^\rho ) } \le C  \|h\| _{ L^\mu (0,1) } \delta ^\alpha M  \label{fFP4} \\
  \|\nabla  \Ical (v) \| _{ L^a((0,1), L^\rho ) } \le C  \|h\| _{ L^\mu (0,1) } \delta ^{\alpha +1} \label{fFP5}   \\   \| \Ical (v) - \Ical (w) \| _{ L^a((0,1), L^\rho ) } \le  C  \|h\| _{ L^\mu (0,1) } \delta ^\alpha  \dist( v, w) .\label{fFP6} 
\end{gather} 
Next, we estimate the weighted norm. We observe that 
\begin{equation} \label{fFP7} 
x e^{i\tau  \Delta } = e^{i \tau \Delta } (x-2i \tau \nabla )
\end{equation} 
for all $ \tau \in \R$. Therefore
\begin{equation} \label{fFP8} 
\begin{split} 
x \Ical (v) (t)  =  &  \int _0^t h(s) e^{i (t-s) \Delta } (x-2i (t-s) \nabla )  |v|^\alpha v\\
= & \int _0^t h(s) e^{i (t-s) \Delta } x |v|^\alpha v - 2i  \int _0^t h(s)(t-s) e^{i (t-s) \Delta } \nabla ( |v|^\alpha v) 
\end{split} 
\end{equation} 
and we deduce from~\eqref{fLB13b2}, \eqref{fLB13} and~\eqref{fLB5} that
\begin{equation} \label{fFP9} 
 \| \,  |\cdot | \Ical (v) \| _{ L^a((0,1), L^\rho ) } \le C  \|h\| _{ L^\mu (0,1) } \delta ^\alpha M .
\end{equation} 
We now set
\begin{gather} 
M= 2 \max \{   \|  e^{it \Delta } \DIb \| _{ L^a ((0,1), L^\rho ) },  \|  |\cdot |  e^{it \Delta } \DIb  \| _{ L^a ((0,1), L^\rho ) }\} \label{fFP10}  \\
\delta = 2 \| \nabla v_0 \| _{ L^a ((0,1), L^\rho ) } . \label{fFP11} 
\end{gather} 
It follows from~\eqref{fFP4}, \eqref{fFP9}  and~\eqref{fFP5} that if $\delta $ is sufficiently small, then
\begin{equation*} 
 \| \Ical (v) \| _{ L^a((0,1), L^\rho ) } \le \frac {M} {2}, \quad  \| \,  |\cdot | \Ical (v) \| _{ L^a((0,1), L^\rho ) } \le \frac {M} {2}, \quad \|\nabla  \Ical (v) \| _{ L^a((0,1), L^\rho ) } \le \frac {\delta } {2}.
\end{equation*} 
Applying~\eqref{fFP10}-\eqref{fFP11}  and~\eqref{fFP2}-\eqref{fFP3}, we deduce that $\Phi : \Ens _{ \delta ,M } \to \Ens _{ \delta ,M }$. Moreover, assuming $\delta $ possibly smaller, it follows from~\eqref{fFP6} that $\Phi $ is a strict contraction on $ \Ens _{ \delta ,M }$. By Banach's fixed point theorem, $\Phi $ has a fixed point $v\in \Ens _{ \delta ,M }$, which is a solution of~\eqref{MNLS}.  

To complete the proof, it remains to show that $v\in C([0,1], X)$.
For this, we observe that by~\eqref{fDfa2} and~\eqref{fDfat} we have $ \widetilde{a} \le \gamma $, so that $ \widetilde{a} ' \ge \gamma ' $.
Therefore, estimates~\eqref{fLB13b1}, \eqref{fLB13b2} and~\eqref{fLB13}, and the fact that $v\in \Ens _{ \delta ,M }$ imply that
\begin{gather*} 
 \| h   |v|^\alpha v \|  _{ L^{ \gamma ' }( (0,1) , L^{ \rho '}) } \le   C \delta ^{\alpha }  M   \\
  \| h  |\cdot |\,   |v|^\alpha v \|  _{ L^{ \gamma ' }( (0,1) , L^{ \rho '}) } \le   C  \delta ^{\alpha }  M  \\
  \| \nabla (h   |v|^\alpha v ) \|  _{ L^{ \gamma ' }( (0,1) , L^{ \rho '}) } \le   C \delta ^{\alpha +1}  
\end{gather*} 
It now follows from the standard Strichartz estimates (i.e., with admissible pairs, see e.g.~\cite[Theorem~2.2.3~(ii)]{CLN10}) that $\Ical (v) \in C ([0, 1], X) $.
Since $v= \Ical (v) + e^{it \Delta } \DIb $ and $\DIb \in X$, 
this completes the proof. 
\end{proof} 

\begin{rem} \label{eGE2} 
The conditions~\eqref{fFP14}--\eqref{fFP17} are satisfied under some stronger, but more familiar, conditions. Indeed, set
\begin{equation} \label{fDfs} 
s = N \Bigl( \frac {1} {2} - \frac {1} {\rho } \Bigr) - \frac {2} {a}
\end{equation} 
so that $0 \le s< 1$ by~\eqref{fDfa2} and~\eqref{fLB8b3}. 
Setting $\frac {1} { \widetilde{\rho } } = \frac {1} {\rho } + \frac {s} {N}$, we see that  $(a, \widetilde{\rho } )$ is an admissible pair, so that by Strichartz's estimates
(see e.g.~\cite[Theorem~2.2~(i)]{CazenaveW2})
$ \|  e^{it\Delta } \varphi  \|_{ L^a (\R, \dot H^{s,  \widetilde{\rho } } ) } \le C  \| \varphi  \| _{ \dot H^s } $,
where $\dot H^{s, p}$ and $\dot H^s= \dot H^{s,2}$ are the homogeneous Sobolev spaces (see e.g.~\cite[Section~6.3]{BerghL}).
Using Sobolev's embedding, we deduce that
\begin{equation*} 
 \|  e^{it\Delta }\varphi  \|_{ L^a (\R, L^\rho ) } \le C  \| \varphi  \| _{ \dot H^s } .
\end{equation*} 
Therefore, conditions~\eqref{fFP14}--\eqref{fFP17}  are satisfied provided $\DIb \in X$, $\DIb \in \dot H^s (\R^N ) $, $ |\cdot |\DIb \in \dot H^s (\R^N ) $, $\nabla \DIb \in \dot H^s (\R^N ) $, and the smallness condition is on the norm $ \| \nabla\DIb \| _{ \dot H^s } $. 
\end{rem} 

\begin{proof} [Proof of Theorem~$\ref{eGE1}$]
Let $\DI$ be as in Theorem~\ref{eGE1}. In particular, $\DIb$ defined by $\DIb (x)= e^{i \frac { |x|^2} {4}} \DI (x)$ satisfies $\DIb\in X$. Moreover, if $s$ is defined by~\eqref{fDfs}, then $s\in [0,1)$, so that $\DIb \in \dot H^s (\R^N ) $, $ |\cdot |\DIb \in \dot H^s (\R^N ) $, $\nabla \DIb \in \dot H^s (\R^N ) $. Therefore, it follows from Remark~\ref{eGE2} that $ e^{it \Delta } \DIb \in  L^a ((0,1), L^\rho ( \R^N ))$, $ |\cdot | e^{it \Delta } \DIb \in  L^a ((0,1), L^\rho ( \R^N ))$, $ e^{it \Delta } \nabla \DIb \in  L^a ((0,1), L^\rho ( \R^N ))$ and $  \|e^{it \Delta } \nabla  \DIb\|_{  L^a ((0,1), L^\rho )} \le  \| \DIb \| _{ H^2 }$. Thus we see that if $\| \DIb \| _{ H^2 }$ is sufficiently small, then $\DIb $ satisfies the assumptions of Theorem~\ref{eGE3}. Let $v\in C([0,1], X)$ be the corresponding solution of~\eqref{MNLS2}. 
Following~\cite{CazenaveW1}, we apply the pseudo-conformal transformation.
More precisely, let the variables $(s,y)\in [0,\infty ) \times \R^N $ be defined by 
\begin{equation*} 
s={\frac {t} {1-t}},\quad y=\frac {x} {1-t}
\end{equation*} 
for $(t,x)\in [0,1) \times \R^N $.
We define $u$ on $[0,\infty ) \times \R^N $ by
\begin{equation*} 
u(s, y) = (1-t)^{N/2} e^{i{\frac {|x|^2} {4(1-t)}}} v(t,x) .
\end{equation*} 
It follows that $u\in C([0,\infty ), X )$, and is a solution of~\eqref{MNLS3} on $[0,\infty )$.
Finally, since $v \in C ([0, 1], X) $, it follows from Proposition~3.14 in~\cite{CazenaveW1}  that  there exists $u^+ \in X$ such that $e^{-is \Delta } u(s) \to u^+ $ in $X$  as $s\to \infty $. This completes the proof.
\end{proof} 

\begin{rem} 
The conditions on the initial value $\DI$ in Theorem~\ref{eGE1} are stronger than the conditions that are actually used in the proof. These latter conditions are expressed in terms of $\DIb = e^{i \frac { |x|^2} {4}} \DI $ in formulas~\eqref{fFP14}--\eqref{fFP17}. 
Intermediate conditions are given in Remark~\ref{eGE2}.
\end{rem}

\end{document}